\documentclass[10pt, reqno]{amsart}
\usepackage{amsmath, amsthm, amscd, amsfonts, amssymb, graphicx, color}
\usepackage[bookmarksnumbered, colorlinks, plainpages]{hyperref}
\hypersetup{colorlinks=true,linkcolor=red, anchorcolor=green, citecolor=cyan, urlcolor=red, filecolor=magenta, pdftoolbar=true}
\usepackage{mathrsfs}
\usepackage[OT2,OT1]{fontenc}
\newcommand\cyr
{
\renewcommand\rmdefault{wncyr}
\renewcommand\sfdefault{wncyss}
\renewcommand\encodingdefault{OT2}
\normalfont
\selectfont
}
\DeclareTextFontCommand{\textcyr}{\cyr}

\newtheorem{theorem}{Theorem}[section]
\newtheorem{lemma}[theorem]{Lemma}

\newtheorem{corollary}[theorem]{Corollary}
\theoremstyle{definition}
\newtheorem{definition}[theorem]{Definition}

\newtheorem{conjecture}[theorem]{Conjecture}

\theoremstyle{remark}
\newtheorem{remark}[theorem]{Remark}
\numberwithin{equation}{section}

\begin{document}
\setcounter{page}{1}

\title[Artin $L$-functions and  noncommutative tori]{Artin $L$-functions and  noncommutative tori}

\author[Nikolaev]
{Igor V. Nikolaev$^1$}

\address{$^{1}$ Department of Mathematics and Computer Science, St.~John's University, 8000 Utopia Parkway,  
New York,  NY 11439, United States.}
\email{\textcolor[rgb]{0.00,0.00,0.84}{igor.v.nikolaev@gmail.com}}

%\dedicatory{In memory of Ola Bratteli}



%\date{Received:  August 14, 2015; Revised: yyyyyy; Accepted: zzzzzz.}

\begin{abstract}
Using the ideas of Deninger,  we prove that the Artin $L$-functions 
coincide with such of the noncommutative tori. 
This result can be viewed as the Langlands reciprocity for 
noncommutative tori. 
\end{abstract}

\maketitle

%**************************************************************************
\section{Introduction}
%***************************************************************************
The aim of our  note is  a conjecture  \cite[Conjecture 1]{Nik1} that every Artin $L$-function [Artin 1924] \cite{Art1}
can be written as the Euler product of the local zeta functions associated to the $2n$-dimensional noncommutative tori [Rieffel 1990] \cite{Rie1}. 
This conjecture is known to be true for $n=0$ and $n=1$ \cite[Theorem 1]{Nik1}.   
Here we prove  the general   case $n\ge 1$ (Theorem \ref{thm1.5}).   Our proof is based on 
the ideas of  [Deninger 2001] \cite{Den1}.

An $m$-dimensional noncommutative torus is the
universal $C^*$-algebra  $\mathscr{A}_{\Theta}^m$  generated by the unitary operators $u_1,\dots, u_m$
satisfying the commutation relations $u_ju_i=e^{2\pi i \theta_{ij}} u_iu_j, ~ 1\le i,j\le m$
for a skew-symmetric matrix  $T=(\theta_{ij})\in M_m(\mathbf{R})$
[Rieffel 1990] \cite{Rie1}. 
Recall that if $m=2n$ then every generic $T$ can be brought by the orthogonal linear 
transformations to the normal form 
$\mathbf{diag} ~\left(\begin{smallmatrix} 0 &\theta_1 \cr -\theta_1 & 0\end{smallmatrix}\right)\dots \left(\begin{smallmatrix} 0 &\theta_n \cr -\theta_n & 0\end{smallmatrix}\right)$,
where $\theta_i>0$ are linearly independent over $\mathbf{Q}$. 
The noncommutative torus is said to have real multiplication,   
if all $\theta_i$  are algebraic numbers
 \cite[Section 1.1.5.2]{N}; we write $\mathscr{A}_{RM}^{2n}$ in this case.

Given the  $\mathscr{A}_{RM}^{2n}$,  one can think of $\theta_i$ as components of 
 a normalized eigenvector $(1,\theta_1,\dots,\theta_{n})$ corresponding to the 
 Perron-Frobenius eigenvalue $\lambda_B>1$ of a positive matrix 
 $B\in GL_{n+1}(\mathbf{Z})$.
 For each prime $p$ we consider the characteristic polynomial $\det\left( xI-B^{\pi(p)}\right)=x^{n+1}-a_1x^{n}-\dots-a_{n}x-1\in\mathbf{Z}[x]$,
 where $\pi(p)$ is an integer-valued function of $p$ \cite[Section 6.6.1]{N}. 
 %****************************************************
 \begin{definition}\label{dfn1.1}
 The local zeta of  the noncommutative torus  $\mathscr{A}_{RM}^{2n}$ 
 is a function:
%****************************************************************************
\begin{equation}\label{eq1.1}
\zeta_p(\mathscr{A}_{RM}^{2n}, z):= {1\over 1-a_1z+a_2z^2-\dots-a_n z^n +pz^{n+1}}, \quad z\in\mathbf{C}.
\end{equation}
%************************************************************************* 
 \end{definition}
 %*****************************************************

 \bigskip
 %**************************************************************
\begin{remark}\label{rmk1.2}
A motivation for (\ref{eq1.1}) is as follows. Let $k$ be a totally real field
and denote by $k_{CM}$ ($k_{RM}$, resp.) its complex (real, resp.) multiplication 
field,  i.e. a totally imaginary (totally real, resp.)  quadratic extension of $k$. 
Then the minimal polynomial of  $k_{CM}$ is an alternating sum of monomials
of the minimal polynomial of the field $k_{RM}$.  Indeed, since both $k_{RM}$
 and $k_{CM}$ are quadratic extensions, their minimal polynomials consist of the monomials 
of even degrees, i.e. $p(x)=x^{2n}+a_2x^{2n-2}+\dots+a_{2n-2}x^2+a_{2n}$. 
Moreover, if $x$ is a real root of $p(x)$ then $ix$ is a complex root 
of the polynomial $q(x)=x^{2n}-a_2x^{2n-2}+\dots+(-1)^{n+1}a_{2n-2}x^2+(-1)^na_{2n}$,
i.e. an alternating sum of the monomials of $p(x)$. 
Thus (\ref{eq1.1}) is an alternating sum of monomials of the minimal polynomial
of an algebraic integer of the field $k_{RM}$ of norm $p$, where $x=z^{-1}$. 
We refer the reader to \cite[Remark 6.6.3]{N} for the explicit matrix formulas.   
\end{remark}
%*************************************************************
%***************************************************************************************
\begin{definition}
An  $L$-function of  the noncommutative torus  $\mathscr{A}_{RM}^{2n}$ 
is  the product of the local zeta functions taken over all but a finite number of primes $p$:
%****************************************************************************
\begin{equation}
L(\mathscr{A}_{RM}^{2n},s):=\prod_p
~\zeta_p(\mathscr{A}_{RM}^{2n}, p^{-s}), \quad s\in \mathbf{C}.
\end{equation}
%************************************************************************* 
\end{definition}
%******************************************************************************************

\bigskip
 Recall that the fundamental problem in algebraic number theory is to describe how an ordinary
prime $p$ factors into prime ideals $\mathfrak{P}$ in the ring of integers of an arbitrary finite extensions
$\mathbf{K}$ of the rational field $\mathbf{Q}$.    Let $O_{\mathbf{K}}$ be the ring of integers of the extension $\mathbf{K}$
and $pO_{\mathbf{K}}$ a principal ideal;  it is known that $pO_{\mathbf{K}}=\prod \mathfrak{P}_i$,
where $\mathfrak{P}_i$ are prime ideals of $O_{\mathbf{K}}$.  If $\mathbf{K}$ is the Galois extension of $\mathbf{Q}$
and $Gal~(\mathbf{K}|\mathbf{Q})$ is the corresponding Galois group,  then each automorphism 
$g\in Gal~(\mathbf{K}|\mathbf{Q})$ ``moves around''  the ideals $\mathfrak{P}_i$ in the prime decomposition
of $p$ over $\mathbf{K}$.  An  isotropy subgroup of $Gal~(\mathbf{K}|\mathbf{Q})$ (for given $p$)  consists of 
the  elements of    $Gal~(\mathbf{K}|\mathbf{Q})$ which fix all the ideals $\mathfrak{P}_i$.   For simplicity,
we shall assume that $p$ is  unramified in $\mathbf{K}$, i.e. all $\mathfrak{P}_i$ are distinct;
in this case the isotropy  subgroup is cyclic.   The (conjugacy class of)  generator
in the cyclic isotropy subgroup of $Gal~(\mathbf{K}|\mathbf{Q})$ corresponding to $p$ is called 
the  Frobenius element and denoted by $Fr_p$.   The element $Fr_p\in Gal~(\mathbf{K}|\mathbf{Q})$ 
describes completely the factorization of $p$ over $\mathbf{K}$ and the major problem
is to express $Fr_p$ in terms of  arithmetic of the ground field $\mathbf{Q}$.  
To handle this  problem,  it was suggested in [Artin 1924] \cite{Art1} to consider an $m$-dimensional
irreducible representation 
$\sigma_m:  Gal~(\mathbf{K}|\mathbf{Q})\longrightarrow GL_m(\mathbb{C})$.
 The idea is to use the characteristic polynomial {\bf Char}  
$(\sigma_m(Fr_p)):=\det (I_m-\sigma_m(Fr_p)z)$ of the 
matrix $\sigma_m(Fr_p)$;  the polynomial is independent of the similarity class 
of   $\sigma_m(Fr_p)$ in the group $GL_m(\mathbb{C})$ and thus provides
an  intrinsic description of the Frobenius element $Fr_p$.  
%************************************************************************************
\begin{definition}
An Artin $L$-function is  the product taken over all but a finite set of primes $p$:
%*********************************************************************************
\displaymath
L(\sigma_m, s):= \prod_p ~\zeta_p(\sigma_m, p^{-s}), ~\hbox{where}  ~\zeta_p(\sigma_m, z)=  \mathbf{Char}  
(\sigma_m(Fr_p)),   ~s\in \mathbf{C}. 
\enddisplaymath
%***********************************************************************************
\end{definition}
%**************************************************************************************

\bigskip
 Our main result can be  formulated as follows. 
 %**************************************************************
\begin{theorem}\label{thm1.5}
{\bf (\cite[Conjecture 6.6.1]{N})}
%*******************************************************************************************
For each  finite Galois extension $\mathbf{K}$  of the field $\mathbf{Q}$
 and each  irreducible representation 
 $\sigma_{n+1}: Gal~(\mathbf{K} | \mathbf{Q})\to GL_{n+1}({\Bbb C})$,
there exists a noncommutative torus  $\mathscr{A}_{RM}^{2n}$,   such that:
%**********************************************************************
\begin{equation}\label{eq1.3}
L(\sigma_{n+1},s)\equiv L(\mathscr{A}_{RM}^{2n}, s).
\end{equation}
%********************************************************************
\end{theorem}
%*************************************************************

\bigskip
%****************************************************
\begin{remark}
Theorem \ref{thm1.5} can be viewed as an analog of the Langlands
reciprocity for the Artin $L$-functions. Namely, such a reciprocity predicts 
a standard $L$-function at the RHS of (\ref{eq1.3}) coming from an automorphic
cuspidal representation of the adelic algebraic group $GL(n)$  [Gelbart 1984] \cite[Conjecture 1]{Gel1}.
It follows from Theorem \ref{thm1.5} that the even dimensional noncommutative
tori encode critical arithmetic data which can be used in non-abelian class field theory 
and higher reciprocity laws  \cite[Conjecture 1]{Nik1}.  
\end{remark}
%****************************************************

The paper is organized as follows.  A brief review of the preliminary facts is 
given in Section 2. Theorem \ref{thm1.5}
proved in Section 3. Some remarks can be found 
in Section 4.

%**************************************************************************
\section{Preliminaries}
%***************************************************************************
We briefly review Deninger's  foliated spaces,  noncommutative tori and their 
relation to the Drinfeld modules. 
We refer the reader to   [Deninger 2001] \cite{Den1},  [Rieffel 1990] \cite{Rie1},  
and  \cite{Nik2} for a detailed exposition.

%**************************************************************************
\subsection{Deninger's foliated spaces}
%***************************************************************************
%**************************************************************************
\subsubsection{Foliated spaces}
%***************************************************************************
Let $X$ be a smooth $n$-dimensional manifold. Consider a $d$-dimensional 
foliation $\mathcal{F}$ on $X$, i.e. the partition of $X$ into a disjoint union 
of the $d$-dimensional submanifolds (called leaves) immersed in $X$. 
The tangent bundle $T\mathcal{F}$ of a foliation $\mathcal{F}$ 
is the union of the tangent spaces to the leaves. The  $T\mathcal{F}$
is a sub-bundle of the tangent bundle $TX$ of the manifold $X$.  
The differential forms of order $n$ along the leaves of $\mathcal{F}$
are the smooth sections $\mathcal{A}_{\mathcal{F}}^n(X)=\Gamma (X, ~\Lambda^nT^*\mathcal{F})$,
where $\Lambda^nT^*\mathcal{F}$ is a real vector bundle of dimension $n$. 
The exterior derivative along the leaves of $\mathcal{F}$ is written 
as $d^n_{\mathcal{F}}: \mathcal{A}_{\mathcal{F}}^n(X)\to \mathcal{A}_{\mathcal{F}}^{n+1}(X)$
and satisfies the relation  $d^{n+1}_{\mathcal{F}}\circ d^n_{\mathcal{F}}=0$. 
%********************************************************************
\begin{definition}
The reduced leafwise cohomology of  foliation $\mathcal{F}$ is given by the formula:
%***********************************************
\begin{equation}\label{eq2.1}
\overline{H}^n_{\mathcal{F}}=Ker ~d^n_{\mathcal{F}}/\overline{Im ~d^{n-1}_{\mathcal{F}}}, 
\end{equation}
%***********************************************
where $\overline{Im ~d^{n-1}_{\mathcal{F}}}$ is the topological closure of 
$Im ~d^{n-1}_{\mathcal{F}}$ in the Fr\'echet topology on the set $\mathcal{A}_{\mathcal{F}}^n(X)$.
\end{definition}
%****************************************************************
%*************************************************************
\begin{remark}\label{rmk2.2}
The cohomology $\overline{H}^n_{\mathcal{F}}$ can (and will be) infinite dimensional even 
if the leaves of $\mathcal{F}$ are dense in $X$.
\end{remark}
%***************************************************************
Consider a smooth flow $\phi:\mathbf{R}\times X\to X$ on the foliated manifold $X$
given by the formula $(t,x)\mapsto \phi^t(x)$ for all $t\in\mathbf{R}$ and $x\in X$. 
The $\phi^t$ is assumed leaf-preserving and transversal to the leaves of foliation $\mathcal{F}$.  
The action $\phi^t_{*}$ of $\phi^t$ on $\overline{H}^n_{\mathcal{F}}$ is linear
for all $n$ and the infinitesimal generator of  $\phi^t_{*}$ is denoted by $\Theta$.

Let $\mathcal{F}$ be a foliation of codimension $1$ on the manifold $X$. 
The flow $\phi$ defines a riemannian metric on $X$ which turns $\overline{H}^n_{\mathcal{F}}$
into a Hilbert space $\hat{H}^n_{\mathcal{F}}$. The linear operator $\phi^t_{*}$ on  $\overline{H}^n_{\mathcal{F}}$
extends to a bounded operator on $\hat{H}^n_{\mathcal{F}}$.
%*****************************************************************************
\begin{theorem}\label{thm2.3}
{\bf \cite[Thm 3.3]{Den1}}
For every test function $\varphi\in C^{\infty}_0(\mathbf{R})$ the 
$A_{\varphi}:=\int_{\mathbf{R}}\varphi(t)\phi_{*}^tdt$ is a trace 
class operator on the Hilbert space  $\hat{H}^n_{\mathcal{F}}$.
Setting 
\linebreak
$Tr~(\phi_{*}~|~\overline{H}^n_{\mathcal{F}}(X))(\varphi)=Tr~A_{\varphi}$
defines a distribution on $\mathbf{R}$,  so that:
%******************************************************
%*******************************************************************
\begin{equation}\label{eq2.2}
\begin{array}{cc}
\sum_{n=0}^{\dim \mathcal{F}} (-1)^n Tr~(\phi_{*}~|~\overline{H}^n_{\mathcal{F}}(X))(\varphi) &= 
\sum_{\gamma} l(\gamma) \sum_{k\in\mathbf{Z}-0}\epsilon_{\gamma}(k)\delta_{kl(\gamma)}+\\
&\\
&+\chi_{C_0}(\mathcal{F},\mu)\delta_0, 
\end{array}
\end{equation}
%********************************************************* 
where $l(\gamma)$ is the length of closed geodesic $\gamma$,
 $\epsilon_{\gamma}(k)=\textnormal{sgn}\det (1-T_x\phi^{kl(\gamma)}~|~T_x\mathcal{F})$,
 $\delta_x$ is the Dirac delta function
 and $\chi_{C_0}(\mathcal{F},\mu)$ is the Connes' Euler characteristic of $\mathcal{F}$ 
 with respect to the transverse measure $\mu$.   
\end{theorem}
%*****************************************************************************

%**************************************************************************
\subsubsection{Dedekind zeta function}
%***************************************************************************
Let $\mathbf{K}$ be a number field. Consider the Dedekind zeta function:
%*********************************************************************
\begin{equation}
\zeta_{\mathbf{K}}(s)=\sum_{I\subseteq O_{\mathbf{K}}} \frac{1}{(N_{\mathbf{K}|\mathbf{Q}}(I))^s},
\end{equation}
%******************************************************************
 where $I$ runs through through the non-zero ideals of the ring of integers
 $O_{\mathbf{K}}$ and  $N_{\mathbf{K}|\mathbf{Q}}(I)$ is the norm of $I$. 
 We shall be using the following explicit formula (due to A.~Weil) 
 for $\zeta_{\mathbf{K}}(s)$, i.e. a relation between sums over the 
 non-trivial zeros $\rho$ of $\zeta_{\mathbf{K}}(s)$ and sums over prime 
 ideals of the field $\mathbf{K}$.  
 %*****************************************************************************
\begin{theorem}
For every test function $\varphi\in C^{\infty}_0(\mathbf{R})$ the 
following equality holds:
%*******************************************************************
\begin{equation}\label{eq2.4}
\begin{array}{cc}
1-\sum_{\rho} e^{t\rho}+e^t &= 
\sum_{\mathfrak{p}<\infty} \log N\mathfrak{p} \sum_{k=1}^{\infty}\delta_{k \log N\mathfrak{p}}+\\
&\\
&+\sum_{\mathfrak{p}=\infty}\frac{1}{1-e^{\kappa_{\mathfrak{p}}t}}, 
\end{array}
\end{equation}
%********************************************************* 
where $\kappa_{\mathfrak{p}}\in \{-2,-1\}$ and  $\delta_x$ is the Dirac delta function. 
\end{theorem}
%*****************************************************************************

%**************************************************************************
\subsubsection{Analogy \cite[Section 4]{Den1}}
%***************************************************************************
Similar to the case of varieties over finite fields, Deninger suggested that the 
leafwise reduced cohomology (\ref{eq2.1}) is underlying an analog of the Lefschetz trace formula
for the varieties over number fields [Deninger 2001] \cite{Den1}.  Unlike the case of finite fields,
such a cohomology is no longer finite dimensional, see Remark \ref{rmk2.2}.  
Comparing formulas (\ref{eq2.2}) and (\ref{eq2.4}), the following beautiful analogy
can be observed.
%*********************************************************************** 
\begin{conjecture}\label{cnj2.5}
Finite places $\mathfrak{p}$ in the explicit formula (\ref{eq2.4}) are closed geodesics 
$\gamma_{\mathfrak{p}}$ in formula (\ref{eq2.2}), such that:
%************************************************************************
\begin{equation}\label{eq2.5}
l(\gamma_{\mathfrak{p}})=\log N\mathfrak{p}. 
\end{equation}
%*********************************************************************** 
\end{conjecture}
%*******************************************************************
%******************************************************************
\begin{remark}
{\bf (\cite{Den1})} 
To understand number theory in geometric terms, it is important to construct a foliated 
space $(X,\phi^t,\mathcal{F})$ which realizes formula (\ref{eq2.5}). 
\end{remark}
%*******************************************************************

%**************************************************************************
\subsection{Noncommutative tori}
%***************************************************************************
%**************************************************************************
\subsubsection{$C^*$-algebras}
%***************************************************************************
The $C^*$-algebra is an algebra  $\mathscr{A}$ over $\mathbf{C}$ with a norm 
$a\mapsto ||a||$ and an involution $\{a\mapsto a^* ~|~ a\in \mathscr{A}\}$  such that $\mathscr{A}$ is
complete with  respect to the norm, and such that $||ab||\le ||a||~||b||$ and $||a^*a||=||a||^2$ for every  $a,b\in \mathscr{A}$.  
Each commutative $C^*$-algebra is  isomorphic
to the algebra $C_0(X)$ of continuous complex-valued
functions on some locally compact Hausdorff space $X$. 
Any other  algebra $\mathscr{A}$ can be thought of as  a noncommutative  
topological space. 

%**************************************************************************
\subsubsection{K-theory of $C^*$-algebras}
%**********************************************************************
By $M_{\infty}(\mathscr{A})$ 
one understands the algebraic direct limit of the $C^*$-algebras 
$M_n(\mathscr{A})$ under the embeddings $a\mapsto ~\mathbf{diag} (a,0)$. 
The direct limit $M_{\infty}(\mathscr{A})$  can be thought of as the $C^*$-algebra 
of infinite-dimensional matrices whose entries are all zero except for a finite number of the
non-zero entries taken from the $C^*$-algebra $\mathscr{A}$.
Two projections $p,q\in M_{\infty}(\mathscr{A})$ are equivalent, if there exists 
an element $v\in M_{\infty}(\mathscr{A})$,  such that $p=v^*v$ and $q=vv^*$. 
The equivalence class of projection $p$ is denoted by $[p]$.   
We write $V(\mathscr{A})$ to denote all equivalence classes of 
projections in the $C^*$-algebra $M_{\infty}(\mathscr{A})$, i.e.
$V(\mathscr{A}):=\{[p] ~:~ p=p^*=p^2\in M_{\infty}(\mathscr{A})\}$. 
The set $V(\mathscr{A})$ has the natural structure of an abelian 
semi-group with the addition operation defined by the formula 
$[p]+[q]:=\mathbf{diag}(p,q)=[p'\oplus q']$, where $p'\sim p, ~q'\sim q$ 
and $p'\perp q'$.  The identity of the semi-group $V(\mathscr{A})$ 
is given by $[0]$, where $0$ is the zero projection. 
By the $K_0$-group $K_0(\mathscr{A})$ of the unital $C^*$-algebra $\mathscr{A}$
one understands the Grothendieck group of the abelian semi-group
$V(\mathscr{A})$, i.e. a completion of $V(\mathscr{A})$ by the formal elements
$[p]-[q]$.  The image of $V(\mathscr{A})$ in  $K_0(\mathscr{A})$ 
is a positive cone $K_0^+(\mathscr{A})$ defining  the order structure $\le$  on the  
abelian group  $K_0(\mathscr{A})$. The pair   $\left(K_0(\mathscr{A}),  K_0^+(\mathscr{A})\right)$
is known as a dimension group of the $C^*$-algebra $\mathscr{A}$.

%**************************************************************************
\subsubsection{Noncommutative tori}
%**********************************************************************
A  noncommutative $m$-torus is the universal $C^*$-algebra
generated by $k$ unitary operators $u_1,\dots, u_m$; the operators
do not commute with each other, but their commutators 
$u_iu_ju_i^{-1}u_j^{-1}$ are fixed scalar multiples 
\linebreak
$\{\exp~(2\pi i\theta_{ij}), ~|~\theta_{ij}\in \mathbf{R}\}$ of the
identity operator.  The $m$-dimensional noncommutative torus, $\mathscr{A}_{T}^m$,
is defined by a skew symmetric real matrix $T=(\theta_{ij}),
~1\le i,j\le k$.   It is known that 
  $K_0(\mathscr{A}_{T}^m)\cong K_1(\mathscr{A}_{T}^m)\cong \mathbf{Z}^{2^{m-1}}$.
The canonical trace $\tau$ on the $C^*$-algebra
$\mathscr{A}_{T}^m$ defines a homomorphism from 
$K_0(\mathscr{A}_{T}^m)$ to the real line $\mathbf{R}$;
under the homomorphism, the image of $K_0(\mathscr{A}_{T}^m)$
is a $\mathbf{Z}$-module, whose generators $\tau=(\tau_i)$ are polynomials 
in $\theta_{ij}$.  The noncommutative
tori  $\mathscr{A}_{T}^m$ and $\mathscr{A}_{T'}^m$ are Morita
equivalent,  if  the matrices $T$ and $T'$
belong to the same orbit of a subgroup $SO(m,m~|~\mathbf{Z})$ of the
group $GL_{2m}(\mathbf{Z})$, which acts on $T$ by the formula
$T'=(AT+B)~/~(CT+D)$, where $(A, B,  C,  D)\in GL_{2m}(\mathbf{Z})$
and  the matrices $A,B,C,D\in GL_m(\mathbf{Z})$ satisfy the conditions:
%*******************************************************************
\begin{equation}%\label{eq1}
A^tD+C^tB=I,\quad A^tC+C^tA=0=B^tD+D^tB,
\end{equation}
%********************************************************************  
where $I$ is the unit matrix and $t$ at the upper right of a matrix 
means a transpose of the matrix.)  
The group $SO(m, m ~| ~\mathbf{Z})$ can be equivalently defined as a
subgroup of the group  $SO(m, m ~| ~\mathbf{R})$ consisting of linear transformations 
of the space $\mathbf{R}^{2m}$,  which 
preserve the quadratic form $x_1x_{m+1}+x_2x_{k+2}+\dots+x_kx_{2m}$.

%**************************************************************************
\subsection{Drinfeld modules}
%***************************************************************************
Let $A=\mathbf{F}_q[t]$ be the ring of polynomials  in one variable over a finite field $\mathbf{F}_q$
and  $k=\mathbf{F}_q(t)$  its field of rational functions  [Rosen 2002] \cite[Chapter 1]{R}. 
Recall that a polynomial $f\in k[x]$ is called additive
in the ring $k[x,y]$ if $f(x+y)=f(x)+f(y)$.  When $char ~k=p$
 the polynomial $\tau_p(x)=x^p$ is additive and each
additive polynomial has the form $a_0x+a_1x^p+\dots+a_rx^{p^r}$. 
The set of all additive polynomials is closed under addition and composition 
operations thus generatng a ring of  the non-commutative polynomials $k\langle\tau_p\rangle$
defined by  the commutation relation $\tau_p a=a^p\tau_p$
for all  $a\in k$. 
The  Drinfeld module $Drin_A^r(k)$  of rank $r\ge 1$ is a homomorphism $\rho$:
%********************************************************************
\begin{equation}\label{eq2.7}
 A\buildrel r\over\longrightarrow k\langle\tau_p\rangle,
\end{equation}
%****************************************************************
given by a polynomial $\rho_a=a+c_1\tau_p+c_2\tau_p^2+\dots+c_r\tau_p^r$ with $c_i\in k$ and $c_r\ne 0$, 
such that for all $a\in A$ the constant term of $\rho_a$ is $a$ and 
$\rho_a\not\in k$ for at least one $a\in A$ [Rosen 2002] \cite[p. 200]{R}.
For each non-zero $a\in A$ the function 
field $k\left(\Lambda_{\rho}[a]\right)$  is a Galois extension of $k$,
such that its  Galois group is isomorphic to a subgroup $G$ of the matrix group $GL_r\left(A/aA\right)$,
where   $\Lambda_{\rho}[a]=\{\lambda\in\bar k ~|~\rho_a(\lambda)=0\}$
is a torsion submodule of the non-trivial  Drinfeld module  $Drin_A^r(k)$  [Rosen 2002] \cite[Proposition 12.5]{R}.
Clearly, the abelian extensions correspond to the case $r=1$.

 Let $G$ be a  left cancellative  semigroup generated by $\tau_p$ and all  $a_i\in k$ subject to the commutation relations 
$\tau_p a_i=a_i^p\tau_p$.  Let $C^*(G)$ be the semigroup $C^*$-algebra [Li 2017] \cite{Li1}.  
For a Drinfeld module  $Drin_A^r(k)$  defined  by  (\ref{eq2.7}) we consider a homomorphism of the semigroup $C^*$-algebras:  
 %********************************************************************
\begin{equation}\label{eq1.4}
C^*(A)\buildrel r\over\longrightarrow C^*(k\langle\tau_p\rangle). 
\end{equation}
%****************************************************************
 It follows from (\ref{eq1.4}) that  the Drinfeld modules  $Drin_A^{r}(k)$ are classified by the  
 noncommutative tori $\mathscr{A}_{RM}^{2r}$; we refer the reader to \cite{Nik2}
 for the details.   A  correspondence $F$ between the Drinfeld modules and noncommutative tori is written as 
 $\mathscr{A}_{RM}^{2r}=F(Drin_A^{r}(k))$. 
If $r=1$ then the Drinfeld module (\ref{eq2.7}) plays the r\^ole of an elliptic curve  with complex multiplication  $\mathcal{E}_{CM}$
[Drinfeld 1974] \cite[p. 594]{Dri1}. The set $\Lambda_{\rho}[a]$ consists
of coefficients of the curve  $\mathcal{E}_{CM}$ which are equal to the value  of the Weierstrass $\wp$-function at the torsion points
of lattice $\Lambda_{CM}\subset\mathbf{C}$, where $\mathcal{E}_{CM}\cong \mathbf{C}/\Lambda_{CM}$. 
Likewise, if $r\ge 1$ then the algebra $\mathscr{A}_{RM}^{2r}$ plays the r\^ole of   the curve  $\mathcal{E}_{CM}$
given by  coefficients  $e^{2\pi i \theta_k}$ of the equations for $\mathscr{A}_{RM}^{2r}$. 
 We  therefore get
 $F(\Lambda_{\rho}[a])=\{\log ~(\varepsilon) ~e^{2\pi i\theta_k} ~|~1\le k\le 2r-1\}$; see \cite{Nik2}
 for the notation and details. 
 %***************************************************************
\begin{theorem}\label{thm2.7}
{\bf (\cite{Nik2})}
The following is true:

\medskip
(i) the map $F: Drin_A^{r}(k)\mapsto \mathscr{A}_{RM}^{2r}$ is a functor 
from the category of Drinfeld  modules $\mathfrak{D}$ to a category 
of the noncommutative tori $\mathfrak{A}$,   which maps any pair of isogenous  (isomorphic, resp.) 
modules  $Drin_A^{r}(k), ~\widetilde{Drin}_A^{r}(k)\in \mathfrak{D}$
to a pair of the homomorphic (isomorphic, resp.)  tori  $\mathscr{A}_{RM}^{2r}, \widetilde{\mathscr{A}}_{RM}^{2r}
\in \mathfrak{A}$;  

\smallskip
(ii) $F(\Lambda_{\rho}[a])=\{e^{2\pi i\theta_k+\log\log\varepsilon} ~|~1\le k\le 2r-1\}$,
where $\mathscr{A}_{RM}^{2r}=F(Drin_A^r(k))$ and $\Lambda_{\rho}[a]$ is the 
torsion submodule of the Drinfeld module  $Drin_A^{r}(k)$; 

\smallskip
(iii) the Galois group $Gal \left(\mathbf{k}(e^{2\pi i\theta_k+\log\log\varepsilon})  ~| ~\mathbf{k}\right)\subseteq GL_{r}\left(A/aA\right)$,
where $\mathbf{k}$ is a subfield of the number field $\mathbf{Q}(e^{2\pi i\theta_k+\log\log\varepsilon})$. 
 \end{theorem}
%***************************************************************

\bigskip
%***************************************************************
\begin{remark}\label{rmk2.8}
Theorem \ref{thm2.7} says that   the Galois extensions $\mathbf{K}$ of the field $\mathbf{Q}$
 arise  from the noncommutative tori  $\mathscr{A}_{RM}^{2r}$ 
 via the formula  $\mathbf{K}=\mathbf{Q}(e^{2\pi i\theta_k+\log\log\varepsilon})$. 
\end{remark}
%**************************************************************

%**************************************************************************
\section{Proof of theorem \ref{thm1.5}}
%***************************************************************************
For the sake of clarity, let us outline the main ideas. 
We construct a foliated space  $(X,\phi^t,\mathcal{F})$ satisfying equation (\ref{eq2.5}), see Lemma \ref{lm3.1}. 
This proves  Conjecture \ref{cnj2.5} of  [Deninger 2001] \cite[Section 4]{Den1}.  
The construction is as follows. Recall that the infinitesimal generator $\Theta$ 
is a bounded linear operator on the Hilbert space  $\hat{H}^n_{\mathcal{F}}$,
see Section 2.1.1.  We consider a $C^*$-algebra  $\mathscr{Z}$  generated by $\Theta$
and all operators $\{A_{\varphi}~|~\varphi\in C_0^{\infty}(\mathbf{R})\}$,
see Theorem \ref{thm2.3}.
We prove that  $\mathscr{Z}$ is the maximal abelian subalgebra (masa) 
of a noncommutative  torus $\mathscr{A}_{RM}^{2n}$, such that 
 $\mathscr{A}_{RM}^{2n}\subset \mathscr{Z}\rtimes\mathbf{Z}$, 
 where the crossed product is taken by the Vershik automorphism of 
 the algebra  $\mathscr{Z}$.  We apply Theorem \ref{thm2.7} 
 and Remark \ref{rmk2.8} to  $\mathscr{A}_{RM}^{2n}$,  thus obtaining 
 a number field $\mathbf{K}$.  It is not hard to see, that our construction
 is a functor $(X,\phi^t,\mathcal{F})\mapsto \mathbf{K}$ realizing formula
  (\ref{eq2.5}).     Theorem \ref{thm1.5} follows from (\ref{eq2.5}) 
 by comparing the Euler factors of the functions $L(\sigma_{n+1},s)$ and $L(\mathscr{A}_{RM}^{2n}, s)$.
Let us pass to a detailed argument. 

%***************************************************
\begin{definition}\label{dfn}
Let $T^{n}\cong\mathbf{R}^{n}/\mathbf{Z}^{n}$
be an $n$-dimensional torus. The foliation $\mathcal{F}_A$ 
on $T^n$ is called Anosov's if $\dim \mathcal{F}_A=n-1$ 
and there exists an automorphism  of the lattice $\mathbf{Z}^{n}$ given by a positive integer matrix 
$A\in GL_{2n}(\mathbf{Z})$, such that the eigenvector $v_A\in\mathbf{R}^{2n}$ corresponding to the Perron-Frobenius 
eigenvalue $\lambda_A>1$ of $A$ is normal to every leaf of $\mathcal{F}_A$ on the covering space $\mathbf{R}^n$. 
\end{definition}
%****************************************************

\bigskip
%******************************************************
\begin{lemma}\label{lm3.1}
{\bf (Main lemma)}
Given foliated space $(T^{2n}, \phi^t, \mathcal{F}_A)$ there exists 
a Galois number field $\mathbf{K}$, such that 
%************************************************************************
\begin{equation}\label{eq3.1}
l(\gamma_{\mathfrak{p}})=\log N\mathfrak{p},
\end{equation}
%*********************************************************************** 
where $l(\gamma_{\mathfrak{p}})$ the length 
of a closed geodesic $\gamma_{\mathfrak{p}}$ and  $N\mathfrak{p}$ the 
norm of the prime ideal $\mathfrak{p}$ of the field $\mathbf{K}$. 
 \end{lemma}
%******************************************************
\begin{proof}
(i)  Let $X:=T^{2n}\cong\mathbf{R}^{2n}/\mathbf{Z}^{2n}$ be a smooth torus of dimension $2n$. 
Consider the Anosov foliation $\mathcal{F}_A$  on $T^{2n}$ of codimension $1$. 
The  $\mathcal{F}_A$ is covered by a family $\{l_{\alpha}\}$ of the parallel 
hyperplanes $l_{\alpha}\cong\mathbf{R}^{2n-1}\subset\mathbf{R}^{2n}$. 
Since $\mathcal{F}_A$ is Anosov's foliation, there exists an automorphism 
of the lattice $\mathbf{Z}^{2n}$ given by a positive integer matrix 
$A\in GL_{2n}(\mathbf{Z})$, such that the (scaled) eigenvector $v_A\in\mathbf{R}^{2n}$ corresponding to the Perron-Frobenius 
eigenvalue $\lambda_A>1$ of $A$ is normal to every leaf in the family $\{l_{\alpha}\}$. 
It is easy to see, that the flow $\phi^t$ transversal to $\mathcal{F}_A$ is covered by a flow
 $\tilde{\phi}^t$ on $\mathbf{R}^{2n}$ given by the equation
$d \tilde{\phi}^t /dt=v_A$.

\bigskip
(ii) Let $\Theta$ be the infinitesimal generator of the flow $\phi^t$, 
see Section 2.1.1.  As explained, such a generator gives rise to a 
bounded linear operator (also denoted by $\Theta$) acting on the Hilbert space   $\hat{H}^n_{\mathcal{F}}$,
which is canonically associated  to the foliated space  $(T^{2n}, \phi^t, \mathcal{F}_A)$. 
The operator $\Theta$ is self-adjoint. Let $\varphi\in C_0^{\infty}(\mathbf{R})$ be a test function and 
consider a trace class operator $A_{\varphi}=\int_{\mathbf{R}}\varphi(t)\phi_{*}^tdt$ on  $\hat{H}^n_{\mathcal{F}}$.
It can be verified that $A_{\varphi}$ commutes with $\Theta$ and with each other for all  test functions $\varphi$. 
The commutative $C^*$-algebra generated by $\Theta$ and all $\{A_{\varphi}~|~\varphi\in C_0^{\infty}(\mathbf{R})\}$
is denoted by $\mathscr{Z}$. 
We identify  $\mathscr{Z}\cong C(X)$, where $C(X)$ is the $C^*$-algebra
of complex valued functions on a Hausdorff space $X$ (Gelfand Theorem).

\bigskip
(iii) Consider an AF $C^*$-algebra $\mathbb{A}_{RM}^{2n}$ given by its 
dimension group \cite[Section 3.5.2]{N}:
%**************************************************************
\begin{equation}\label{eq3.2}
{\mathbb Z}^{2n}\buildrel\rm 
A
\over\longrightarrow {\mathbb Z}^{2n}
\buildrel\rm 
A
\over\longrightarrow
{\mathbb Z}^{2n}\buildrel\rm 
A
\over\longrightarrow \dots,
\end{equation}
%*************************************************************** 
where $A\in GL_{2n}(\mathbf{Z})$ is a positive integer  matrix 
corresponding to the foliated space  $(T^{2n}, \phi^t, \mathcal{F}_A)$
as explained in (i). 
Let $\theta=(1,\theta_1,\dots,\theta_{2n-1})$ be the scaled 
eigenvector $v_A$ of matrix $A$ and consider the Bratteli 
compactum $X_{\theta}$ of $\mathbb{A}_{RM}^{2n}$. 
It follows from \cite[Theorem 1 and Remark 1]{Nik3}
that $\mathscr{Z}\cong C(X_{\theta})$ and 
%************************************************************************
\begin{equation}
\mathbb{A}_{RM}^{2n}\cong \mathscr{Z}\rtimes\mathbf{Z},
\end{equation}
%**********************************************************
where the crossed product is taken by the Vershik automorphism 
of  $C(X_{\theta})$.  It is well known that the AF $C^*$-algebra  
 $\mathbb{A}_{RM}^{2n}$ contains a copy of the noncommutative 
 torus  $\mathscr{A}_{RM}^{2n}$ \cite[Theorem 3.5.3]{N}.
 We therefore get a map $(T^{2n},\phi^t,\mathcal{F}_A)\mapsto  \mathscr{A}_{RM}^{2n}$
 which preserves morphisms in the corresponding categories of objects. 
 
 \bigskip
 (iv) We apply Theorem \ref{thm2.7} and Remark \ref{rmk2.8} which say that the 
 noncommutative tori  $\mathscr{A}_{RM}^{2n}$ generate the Galois extensiuons
 $\mathbf{K}$ of the field $\mathbf{Q}$. Thus one gets a composition of maps
 which preserves  morphisms between the objects in the corresponding categories:
  %************************************************************************
\begin{equation}\label{eq3.4}
(T^{2n},\phi^t,\mathcal{F}_A)\mapsto  \mathscr{A}_{RM}^{2n}\mapsto \mathbf{K}.
\end{equation}
%**********************************************************

\bigskip
(v)  The Selberg zeta function $\zeta_{\mathcal{F}_A}(s)$
is given by the formula:
%*****************************************************************************
\begin{equation}\label{eq3.5}
\zeta_{\mathcal{F}_A}(s)=\prod_{\gamma_{\mathfrak{p}}} \frac{1}{1-l(\gamma_{\mathfrak{p}})^{-s}},
\end{equation}
%********************************************************************************
 where $l(\gamma_{\mathfrak{p}})$ are lengths of all  simple closed geodesics $\gamma_{\mathfrak{p}}$
 of  the foliated space $(T^{2n},\phi^t,\mathcal{F}_A)$.
Likewise, the Dedekind zeta function of the number field $\mathbf{K}$ is 
%*****************************************************************************
\begin{equation}
\zeta_{\mathbf{K}}(s)=\prod_{\mathfrak{p}\subseteq O_{\mathbf{K}} } \frac{1}{1 - N (\mathfrak{p} )^{-s} },
\end{equation}
%********************************************************************************
where $\mathfrak{p}$ runs through all prime ideals of the ring of integers $O_{\mathbf{K}}$ of the 
field $\mathbf{K}$.   In view of (\ref{eq3.4}) the zeta functions $\zeta_{\mathcal{F}_A}(s)$ and 
$\zeta_{\mathbf{K}}(s)$ must coincide as well as their explicit formulas (\ref{eq2.2}) and (\ref{eq2.4}).
The necessary condition for the latter is given by the equation (\ref{eq3.1}). 
Lemma \ref{lm3.1} is proved. 
\end{proof}
%****************************************************
%***************************************************
\begin{definition}\label{dfn3.2}
By the zeta function $\zeta_{\mathscr{A}_{RM}^{2n}}(s)$ of the 
noncommutative torus $\mathscr{A}_{RM}^{2n}$ one understands 
the Selberg zeta function (\ref{eq3.5}) associated to the 
foliated space  $(T^{2n},\phi^t,\mathcal{F}_A)\mapsto  \mathscr{A}_{RM}^{2n}$
according to the formula (\ref{eq3.4}).  
\end{definition}
%****************************************************

%******************************************************
\begin{lemma}\label{lm3.2}
$\zeta_{\mathscr{A}_{RM}^{2n}}(s)=\left(L(\mathscr{A}_{RM}^{2n},s)\right)^{n+1}$. 
\end{lemma}
%******************************************************
\begin{proof}
Since $\sigma_{n+1}:Gal~(\mathbf{K}|\mathbf{Q})\to GL_{n+1}(\mathbf{C})$ is an
irreducible representation, one gets [Artin 1924] \cite[p. 96]{Art1}:
%***********************************************************************
\begin{equation}\label{eq3.7}
\zeta_{\mathbf{K}}(s)=\left(L(\sigma_{n+1},s)\right)^{n+1},
\end{equation}
%*********************************************************************
where $L(\sigma_{n+1},s)$ is the Artin $L$-function.  In view of Lemma \ref{lm3.1} and 
equivalence $\zeta_{\mathbf{K}}(s)\equiv \zeta_{\mathscr{A}_{RM}^{2n}}(s)$, we conclude from (\ref{eq3.7}) that: 
%***********************************************************************
\begin{equation}\label{eq3.8}
\zeta_{\mathscr{A}_{RM}^{2n}}(s)=\left(\widetilde{L}(\mathscr{A}_{RM}^{2n},s)\right)^{n+1},
\end{equation}
%*********************************************************************
where $\widetilde{L}(\mathscr{A}_{RM}^{2n},s)$ is an $L$-function with  the Euler product
 constructed  from the noncommutative torus  $\mathscr{A}_{RM}^{2n}$. 
 But all such $L$-functions (\ref{eq1.1})  must coincide except for a finite number of its Euler
 factors $\zeta_p(\mathscr{A}_{RM}^{2n}, z)$. Therefore $\widetilde{L}(\mathscr{A}_{RM}^{2n},s)\equiv
 L(\mathscr{A}_{RM}^{2n},s)$ up to a finite number of the $\zeta_p(\mathscr{A}_{RM}^{2n}, z)$. 
From (\ref{eq3.8}) one gets $\zeta_{\mathscr{A}_{RM}^{2n}}(s)=\left(L(\mathscr{A}_{RM}^{2n},s)\right)^{n+1}$. 
Lemma \ref{lm3.2} follows. 
\end{proof}
%****************************************************

%******************************************************
\begin{corollary}\label{cor3.3}
 $L(\sigma_{n+1},s)\equiv L(\mathscr{A}_{RM}^{2n}, s)$ except for a finite number of the Euler factors. 
\end{corollary}
%******************************************************
\begin{proof}
From lemma \ref{lm3.1} and definition \ref{dfn3.2},  one obtains  \linebreak
$\zeta_{\mathscr{A}_{RM}^{2n}}(s)\equiv\zeta_{\mathbf{K}}(s)=\left(L(\sigma_{n+1},s)\right)^{n+1}$. 
We apply lemma \ref{lm3.2} and  get the equality $L(\sigma_{n+1},s)\equiv L(\mathscr{A}_{RM}^{2n}, s)$.
Corollary \ref{cor3.3} is proved. 
\end{proof}
%****************************************************

\bigskip
Theorem \ref{thm1.5} follows from lemmas \ref{lm3.1}, \ref{lm3.2} and corollary \ref{cor3.3}.

%**************************************************************************
\section{Remarks}
%***************************************************************************
%****************************************************
\begin{remark}
The maximal abelian subalgebra $\mathscr{Z}\subset  \mathscr{A}_{RM}^{2n}$ 
is generated by the linear operator $\Theta$ and trace class operators $A_{\varphi}$ for 
all test functions $\varphi\in C_0^{\infty}(\mathbf{R})$. The spectrum of $\Theta$ is related
to the Riemann Hypothesis for the Dedekind zeta function $\zeta_{\mathbf{K}}(s)$   
 [Deninger 2001] \cite[Corollary 3.5]{Den1}. 
\end{remark}
%****************************************************
%****************************************************
\begin{remark}
The reduced leafwise cohomology $\overline{H}^n_{\mathcal{F}}(X)$ is 
receptacle of the Lefschetz Trace Formula for the varieties over the number fields  
[Deninger 2001] \cite[Theorem 3.3]{Den1}.  It will be interesting to relate the  $\overline{H}^n_{\mathcal{F}}(X)$
with the $K$-theory of the noncommutative tori  $\mathscr{A}_{RM}^{2n}$.  
\end{remark}
%****************************************************

\bibliographystyle{amsplain}

\begin{thebibliography}{99}

\bibitem{Art1}
E.~Artin, \textit{\"Uber eine neue Art von $L$-Reihen}, Abhandlungen aus dem Mathematischen 
 Seminar der Hamburgerischen Universit\"at, {\bf 3} (1924), 89-108.  
 

\bibitem{Den1}
C.~Deninger, \textit{Number theory and dynamical systems on foliated spaces},
Jahrsber. Deutsch. Math.-Verein. {\bf 103} (2001), 79-100. 


\bibitem{Dri1}
V.~G.~Drinfeld, 
 {\cyr  \textit{E1llipticheskie moduli},
Matem. Sbornik {\bf 94} (1974), 594-627.}



\bibitem{Gel1}
S.~Gelbart, \textit{An elementary introduction to the Langlands program},
Bull. Amer. Math. Soc. {\bf 10} (1984), 177-219. 

\bibitem{Li1}
X.~Li, \textit{Semigroup $C^*$-algebras}, 
Operator algebras and applications -- the Abel Symposium 2015, 191-202, Abel Symp., 12, Springer,  2017.



\bibitem{Nik1}
I.~V.~Nikolaev,
\textit{Langlands reciprocity for the even dimensional noncommutative tori}, 
 Proc. Amer. Math. Soc. {\bf 139} (2011), 4153-4162. 
 
 \bibitem{Nik3}
 I.~V.~Nikolaev,
 \textit{On a Poisson summation formula for noncommutative tori},
 Complex Anal. Oper. Theory {\bf 13} (2019), 3117-3125. 
 
  \bibitem{N}
I.~V.~Nikolaev, \textit{Noncommutative Geometry}, Second Edition,
De Gruyter Studies in Math. {\bf 66}, Berlin, 2022.

 
 \bibitem{Nik2}
 I.~V.~Nikolaev, \textit{Non-abelian class field theory and higher dimensional noncommutative tori},
 arXiv:2309.05779

\bibitem{Rie1}
M.~A.~Rieffel, \textit{Non-commutative tori -- a case study of non-commutative
differentiable manifolds},  Contemp. Math. {\bf 105} (1990), 191-211. 


\bibitem{R}
M.~Rosen, \textit{Number Theory in Function Flelds},
GTM {\bf 210}, Springer,  2002. 

\end{thebibliography}

%**********************************************************

\end{document}